\documentclass[12pt]{article}
\usepackage[utf8]{inputenc}
\usepackage{amsthm}
\usepackage{amsmath}
\usepackage{bbm}
\usepackage{amsfonts}
\usepackage{amssymb}
\usepackage[left=2cm,right=2cm,top=2cm,bottom=2cm]{geometry}

\usepackage{hyperref}

\usepackage{tabto}
\TabPositions{2 cm, 4 cm, 6 cm, 8 cm}

\usepackage{tikz-cd}

\usepackage{tikz}
\usetikzlibrary{arrows, automata, positioning}

\usepackage{tocloft}
\usepackage{appendix}

\usepackage{enumerate}

\usepackage{hyperref}
 
 \newtheorem{thmA}{Theorem}

\newtheorem{corA}[thmA]{Corollary}
\newtheorem{lemmaA}[thmA]{Lemma}
\newtheorem{quA}[thmA]{Question}
 
\newtheorem{theorem}{Theorem}[section]
\newtheorem{corollary}[theorem]{Corollary}
\newtheorem{lemma}[theorem]{Lemma}
\newtheorem{claim}[theorem]{Claim}
\newtheorem{proposition}[theorem]{Proposition}

\theoremstyle{definition}
\newtheorem{definition}[theorem]{Definition}
\newtheorem{example}[theorem]{Example}
\newtheorem{remark}[theorem]{Remark}

\newtheorem{question}[theorem]{Question}

\newtheorem*{lemma*}{Lemma}
\newtheorem*{proposition*}{Proposition}
\newtheorem*{theorem*}{Theorem}
\newtheorem*{corollary*}{Corollary}

\newcommand{\Map}{\operatorname{Map}}

\newcommand{\Aut}{\operatorname{Aut}}
\newcommand{\Out}{\operatorname{Out}}
\newcommand{\Inn}{\operatorname{Inn}}
\newcommand{\ad}{\operatorname{ad}}

\newcommand{\scl}{\operatorname{scl}}
\newcommand{\acl}{\operatorname{acl}}
\newcommand{\sacl}{\operatorname{sacl}}
\newcommand{\rot}{\operatorname{rot}}

\newcommand{\HH}{\operatorname{H}}
\newcommand{\EH}{\operatorname{EH}}

\newcommand{\G}{\Gamma}
\newcommand{\gp}{\Gamma\mathcal{G}}

\newcommand{\Zmax}{\mathcal{Z}_{\mathrm{max}}}

\begin{document}

\title{$\Aut$-invariant quasimorphisms on groups}
\author{Francesco Fournier-Facio and Richard D. Wade}
\date{\today}
\maketitle

\begin{abstract}
For a large class of groups, we exhibit an infinite-dimensional space of homogeneous quasimorphisms that are invariant under the action of the automorphism group. This class includes non-elementary hyperbolic groups, infinitely-ended finitely generated groups, some relatively hyperbolic groups, and a class of graph products of groups that includes all right-angled Artin and Coxeter groups that are not virtually abelian.

This was known for $F_2$ by a result of Brandenbursky and Marcinkowski, but is new even for free groups of higher rank, settling a question of Mikl{\'o}s Ab{\'e}rt. The case of graph products of finitely generated abelian groups settles a question of Micha{\l} Marcinkowski. As a consequence, we deduce that a variety of $\Aut$-invariant norms on such groups are unbounded.
\end{abstract}

\section{Introduction}

Let $G$ be a discrete group. A \emph{quasimorphism} (or \emph{quasicharacter}) on $G$ is a function $f : G \to \mathbb{R}$ such that the quantity
$$D(f) := \sup\limits_{g, h} |f(g) + f(h) - f(gh)|$$
called the \emph{defect} of $f$, is finite. A quasimorphism is \emph{homogeneous} (or a \emph{pseudocharacter}) if it restricts to a homomorphism on every cyclic subgroup of $G$.
Quasimorphisms play an important role in various areas of mathematics, such as bounded cohomology \cite{frigerio}, stable commutator length \cite{calegari}, geometric group theory \cite{BF}, knot theory \cite{knot}, symplectic geometry \cite{symplectic} and dynamics \cite{ghys}.

In the study of quasimorphisms, free groups have long been a central case study. As a concrete reason why this case is especially important, homogeneous quasimorphisms on free groups can be used to compute stable commutator length on free groups, which in turn gives upper bounds on the stable commutator length of elements in any group \cite[Chapter 2]{calegari}.
But most of all, free groups have played a central role for historical reasons. A special class of quasimorphisms of free groups, the so-called \emph{Brooks (counting) quasimorphisms}, serve as the first and simplest examples of quasimorphisms that are not perturbations of homomorphisms \cite{brooks, grigorchuk}.
Brooks' construction has been generalized in many different directions. Building on \cite{EF,fuj1,fuj2}, Bestvina and Fujiwara \cite{BF} used a coarse version of Brooks' construction to build quasimorphisms for groups that admit \emph{weakly properly discontinuous} (WPD) actions on Gromov-hyperbolic spaces. This was further weakened to only require a condition called WWPD in \cite{BBF} (see also \cite{HM}). There are also analogues in non-positive curvature: in certain cases one can use this technique to produce quasimorphisms on groups acting on CAT(0) spaces \cite{cat0}, buildings \cite{buildings}, and CAT(0) cube complexes \cite{FFT, median}. \\

Given an automorphism $\phi \in \Aut(G)$, and a quasimorphism $f : G \to \mathbb{R}$, we obtain a quasimorphism $f \circ \phi$. This yields an action of $\Aut(G)$ on the space of quasimorphisms on $G$, which preserves the subspace of homogeneous quasimorphisms. In the presence of a large space of quasimorphisms, this action is poorly understood. This led Mikl{\'o}s Ab{\'e}rt to ask about global fixpoints for this action:

\begin{quA}[{\cite[Question 47]{abert}}]
\label{q}

Let $n \geq 2$ and let $F_n$ be a free group of rank $n$. Do there exist non-zero $\Aut$-invariant homogeneous quasimorphisms on $F_n$?
\end{quA}

He comments ``probably not''. Looking at the first source of quasimorphisms of $F_n$, namely Brooks quasimorphisms and linear combinations thereof, does not help. Indeed, it was shown by Hase that, while the span of (homogenisations of) Brooks quasimorphisms is $\Aut$-invariant \cite{quasiout}, it does not admit any finite-dimensional $\Aut$-invariant subspace \cite{hase}. Using more geometric methods, Brandenbursky and Marcinkowski exhibited an infinite-dimensional space of $\Aut$-invariant homogeneous quasimorphisms on $F_2$ \cite{rank2}.

In another direction, Marcinkowski asked about the existence of $\Aut$-invariant homogeneous quasimorphisms on graph products of finitely generated abelian groups \cite[Problem 5.1]{marcinkowski:graph}. This was achieved in many cases by Karlhofer \cite{karlhofer:graph}, who also constructed $\Aut$-invariant homogeneous quasimorphisms on some free products \cite{karlhofer:free}. \\

In this note, we construct $\Aut$-invariant homogeneous quasimorphisms on many groups, giving a more complete picture:

\begin{thmA}[Theorem \ref{thm:main}]
\label{intro:thm:main}

Let $G$ be a finitely generated group, and suppose that one of the following holds:
\begin{enumerate}
    \item $G$ is a non-elementary Gromov-hyperbolic group;
    \item $G$ is non-elementary hyperbolic relative to a collection of finitely generated subgroups $\mathcal{P}$, and no  group in $\mathcal{P}$ is relatively hyperbolic;
    \item $G$ has infinitely many ends;
    \item $G$ is a graph product of finitely generated abelian groups on a finite graph, and $G$ is not virtually abelian.
\end{enumerate}

Then there exists an infinite-dimensional space of $\Aut$-invariant homogeneous quasimorphisms on $G$.
\end{thmA}

All four items include the free group, which recovers the case of rank $2$ \cite{rank2} and settles Question \ref{q} in rank at least $3$. Item $4$ settles \cite[Problem 5.1]{marcinkowski:graph} and recovers \cite[Theorems 1.1, 1.2]{karlhofer:graph}. Item $2$ or $3$ recovers the main result of \cite{karlhofer:free} for finitely generated groups. While the assumption of finite generation is necessary, combining our results with previous constructions we are able to include free products of finitely many freely indecomposable groups (Corollary \ref{cor:freeprod}), removing a technical condition in \cite[Theorem 1.1]{karlhofer:graph}.  \\

One motivation to study $\Aut$-invariant homogeneous quasimorphisms is that they provide lower bounds for $\Aut$-invariant norms on groups. On this side there has been more work, where unboundedness of such norms was proven for free groups \cite{norm:free}, then free and surface groups \cite{rank2}, and finally for non-virtually abelian graph products of finitely generated abelian groups \cite{marcinkowski:graph}. Theorem \ref{intro:thm:main} allows us to recover and expand on these results:

\begin{corA}[Proposition \ref{prop:wordnorm}, Corollary \ref{cor:aut:scl}]
\label{intro:cor:wordnorm}

Let $G$ be a group as in Theorem \ref{intro:thm:main}. Let $S \subset G$ be a finite subset, and let $S^{\Aut}$ be the closure of $S$ under the action of $\Aut(G)$. Then the word norm $|\cdot|_{S^{\Aut}}$ on $G$ is unbounded.

Moreover, the stable autocommutator length on $G$ is unbounded on $[G, G]$.
\end{corA}

Note that both the word norm and the stable autocommutator length $\sacl$ are allowed to take infinite values; however $\sacl$ is finite on $[G, G]$. Although the first part of Corollary~\ref{intro:cor:wordnorm} as stated only tells us something when $S^{\Aut}$ generates $G$, which is usually the case in practice, in many cases the same ideas also apply to subgroups: see Remark~\ref{rem:subgroups}.

We define $\sacl$ in Definition \ref{defi:aut:scl}: this is a generalization of $\Aut$-invariant stable commutator length that, unlike the one used in \cite{karlhofer:free, karlhofer:graph}, does not require $G$ to have trivial center. Unboundedness of $\sacl$ moreover implies unboundedness of any mixed stable commutator length on $G$, in the sense of \cite{mixed:scl1, mixed:scl2} (see Lemma \ref{lem:mixed}). \\

The constructions of $\Aut$-invariant homogeneous quasimorphisms from \cite{karlhofer:free, karlhofer:graph} are quite distinct from ours, and have the advantage of being very explicit and easy to compute. Our approach is more conceptual, and goes via the following observation, which is already implicit in \cite{rank2}:

\begin{lemmaA}[Lemma \ref{lem:main}]
\label{intro:lem:main}

Let $G$ be a group, and suppose that $\Aut(G)$ admits homogeneous quasimorphisms that do not vanish identically on $\Inn(G)$. Then $G$ admits non-zero $\Aut$-invariant homogeneous quasimorphisms.
\end{lemmaA}

More precisely, Lemma~\ref{lem:main} states that homogeneous quasimorphisms on $\Aut(G)$ pull back to $\Aut$-invariant homogeneous quasimorphisms on $G$. Thus, to prove Theorem \ref{intro:thm:main} for a certain family of groups, it suffices to find an infinite-dimensional space of homogeneous quasimorphisms on $\Aut(G)$ that restrict to distinct  quasimorphisms on $\Inn(G)$. This allows us to prove the following result:

\begin{thmA}[Theorem~\ref{thm:ah}]
\label{intro:thm:ah}

Let $G$ be a group, and suppose that $G$ is not virtually central and $\Aut(G)$ is acylindrically hyperbolic. Then there exists an infinite-dimensional space of $\Aut$-invariant homogeneous quasimorphisms on $G$.
\end{thmA}

Thanks to the work of Genevois and Horbez \cite{auto:infend} on acylindrical hyperbolicity of automorphism groups (see also \cite{auto:oneend}), Theorem~\ref{intro:thm:ah} implies Items $1, 2$ and $3$ of Theorem \ref{intro:thm:main}. To prove Theorem~\ref{intro:thm:ah}, we first apply a result of Osin \cite[Lemma~7.1]{osin} that shows that when $\Aut(G)$ acts acylindrically and non-elementarily on a hyperbolic graph $X$, the restriction of this action to $\Inn(G) \cong G/Z(G)$ is also non-elementary, as soon as $\Inn(G)$ is infinite.
We then show that Bestvina and Fujiwara's construction of quasimorphisms is flexible enough to build homogeneous quasimorphisms on $\Aut(G)$ that are nonvanishing on $\Inn(G)$ (Theorem~\ref{thm:BF}), allowing us to apply Lemma~\ref{intro:lem:main}.   

For graph products, $\Aut(G)$ is acylindrically hyperbolic only in certain cases \cite[Theorem 1.5]{auto:graph}. However the description of the automorphism group of a graph product in \cite{auto:graph} (see also \cite{auto:cyclic}), allows us to prove that graph products admit $\Aut$-invariant homogeneous quasimorphisms more broadly (including Item $4$ of Theorem \ref{intro:thm:main}): see Theorem \ref{thm:graph_product} and Corollary \ref{cor:gp:fga}. \\

As a further application of Theorem \ref{intro:thm:ah}, we use \cite{bin} to obtain groups with an infinite-dimensional space of $\Aut$-invariant homogeneous quasimorphisms that have prescribed outer automorphism group:

\begin{corA}[Corollary~\ref{cor:bin}]
\label{intro:cor:bin}

For every countable group $Q$, there exists a finitely generated acylindrically hyperbolic group $G$ with an infinite-dimensional space of $\Aut$-invariant homogeneous quasimorphisms, such that $\Out(G) \cong Q$.
\end{corA}

In Section \ref{s:outlook} we discuss three directions for future research: extensions to acylindrically hyperbolic groups, with applications to commensurability invariance; the problem of classifying elements of $F_n$ with positive stable autocommutator length; and the existence of finitely generated groups with no $\Aut$-invariant quasimorphisms. \\

\textbf{Outline:} In Section \ref{s:qm} we review basic material on quasimorphisms and prove Lemma \ref{intro:lem:main}. Then in Section \ref{s:norms} we introduce $\Aut$-invariant word norms and stable autocommutator length, and prove Corollary \ref{intro:cor:wordnorm}. In Section \ref{s:ggt} we set up the necessary tools from geometric group theory. We prove our main results in Section \ref{s:proof}. In Section \ref{s:outlook} we discuss three directions for future research. \\

\textbf{Acknowledgements:} The authors wish to thank Benjamin Br{\"u}ck, Koji Fujiwara, Anthony Genevois, Camille Horbez, Morimichi Kawasaki, Micha{\l} Marcinkowski, Alessandro Sisto and Bin Sun for useful conversations. They also wish to thank the anonymous referee for several insightful comments. Wade is funded by The Royal Society through a University Research Fellowship.

\section{Quasimorphisms}
\label{s:qm}

We start by recalling basic notions around quasimorphisms.

We will denote by $Q(G)$ the space of quasimorphisms of $G$, and by $Q_h(G)$ the space of homogeneous quasimorphisms of $G$. These are contravariant functors: given a homomorphism $\phi : G \to H$, there is a linear \emph{pullback} $\phi^* : Q(H) \to Q(G)$, defined by $f \mapsto f \circ \phi$ , which restricts to a pullback $\phi^*:Q_h(H) \to Q_h(G)$. In the special case in which $\phi$ is an inclusion, $\phi^*$ is obtained by simply restricting quasimorphisms to a given subgroup, so it is accordingly referred to as the \emph{restriction} map. \\

Homogeneous quasimorphisms are relevant because they serve as representatives of quasimorphisms modulo bounded functions (which we denote by $\ell^\infty(G)$):

\begin{proposition}[{\cite[Section 2.2]{calegari}}]
\label{prop:homogenisation}

Every quasimorphism is at a bounded distance from a unique homogeneous quasimorphism, called its \emph{homogenisation}. Therefore, the map $Q(G) \to Q(G)/\ell^\infty(G)$ restricts to an isomorphism on $Q_h(G)$.
\end{proposition}

The following fact about homogeneous quasimorphisms will be crucial:

\begin{lemma}[{\cite[Section 2.2]{calegari}}]
\label{lem:conjinv}

Homogeneous quasimorphisms are conjugacy-invariant.
\end{lemma}

Therefore the action of $\Aut(G)$ on $Q_h(G)$ descends to an action of $\Out(G)$. In fact \cite{hase} states the $\Aut$-invariance as $\Out$-invariance. \\

We now prove Lemma \ref{intro:lem:main}, which we state in a more precise form here:

\begin{lemma}
\label{lem:main}

Let $G$ be a group, and let $\ad : G \to \Inn(G) \hookrightarrow \Aut(G)$ denote the natural map. Then for all $f \in Q_h(\Aut(G))$, the quasimorphism $\ad^* f \in Q_h(G)$ is $\Aut$-invariant. It is zero if and only if $f$ vanishes identically on $\Inn(G)$.
\end{lemma}

\begin{proof}
Let $\phi \in \Aut(G)$. We have
$$\ad^* f (\phi(g)) = f (\ad_{\phi(g)}) = f (\phi \ad_g \phi^{-1}) = f(\ad_g) = \ad^* f(g),$$
where we used the formula $\ad_{\phi(g)} = \phi \ad_g \phi^{-1}$, and the fact that homogeneous quasimorphisms are conjugacy-invariant (Lemma \ref{lem:conjinv}).

Thus $\ad^* f$ is $\Aut$-invariant, for every $f \in Q_h(\Aut(G))$. The last statement is obvious.
\end{proof}

\begin{remark}
Since $\Out(F_n)$ admits rich actions on natural hyperbolic graphs, $Q_h(\Out(F_n))$ is infinite-dimensional \cite[Corollary 4.30]{outfncomplex}, and therefore so is $Q_h(\Aut(F_n))$, by pulling back quasimorphisms. This is well-known; however all quasimorphisms obtained via this argument factor through $\Out(F_n)$, so they do not help in constructing $\Aut$-invariant homogeneous quasimorphisms on $F_n$.
\end{remark}

The following variation of Lemma~\ref{lem:main} will also be useful.

\begin{proposition}[cf. \cite{karlhofer:graph}, Lemma~2.9] \label{prop:char}
Let $G$ be a group and let $N \leq G$ be a characteristic subgroup with associated quotient map $\pi: G \to G/N$. Then $\pi^*: Q_h(G/N) \to Q_h(G)$ is injective and sends $\Aut(G/N)$-invariant homogeneous quasimorphisms to $\Aut(G)$-invariant homogeneous quasimorphisms.
\end{proposition}

\begin{proof}
Suppose $\phi \in \Aut(G)$ and $\bar{\phi}$ is the induced automorphism of $G/N$. If $g\in G$ then $\pi(\phi(g))=\bar{\phi}(\pi(g))$, so that if $f$ is $\Aut(G/N)$-invariant we have \[\pi^* f(\phi(g))=f(\pi(\phi(g)))=f(\bar{\phi}(\pi(g)))= f(\pi(g))=\pi^* f (g). \] Hence $\pi^*f$ is $\Aut(G)$-invariant. Injectivity follows from surjectivity of the quotient map $G \to G/N$.
\end{proof}

Let us close this section with a remark on dimensions:

\begin{remark}
The statements of our results give \emph{infinite-dimensional} spaces of homogeneous quasimorphisms. In much of the literature on the topic, the statements give \emph{continuum-dimensional} spaces of homogeneous quasimorphisms (see e.g. \cite{EF, fuj1, fuj2}). However, in most cases the two statements are equivalent. Indeed, for every group $G$, the space $Q_h(G)/H^1(G)$ is a Banach space, whose norm is defined by the defect \cite[Corollary 2.57]{calegari} (here we denote by $H^1(G)$ the space of homomorphisms $G \to \mathbb{R}$). It is a standard fact from functional analysis that an infinite-dimensional Banach space has uncountable dimension (see e.g. \cite[Exercise 3.8.2]{brezis}). Therefore, if $H^1(G)$ is finite-dimensional (for instance if $G$ is finitely generated), then $Q_h(G)$ has uncountable dimension as soon as it is infinite-dimensional. In fact, the Bestvina--Fujiwara construction (see Section \ref{s:ggt}) may be easily modified so that the quasimorphisms constructed are linearly independent also modulo homomorphisms \cite[Proof of Theorem 1]{BF}, so they automatically yield a space of homogeneous quasimorphisms of uncountable dimension.

This discussion also applies to $\Aut$-invariant homogeneous quasimorphisms, since $\Aut$-invariance is a closed condition, and therefore the space of $\Aut$-invariant homogeneous quasimorphisms modulo $\Aut$-invariant homomorphisms is again a Banach space with the defect norm.
\end{remark}

\section{$\Aut$-invariant norms}
\label{s:norms}

One of the main applications of quasimorphisms is in the study of length functions on groups. Most notably, quasimorphisms can be used to compute stable commutator length \cite{calegari}. In our setting, $\Aut$-invariant homogeneous quasimorphisms can be used to prove unboundedness of $\Aut$-invariant word norms, in the following sense:

\begin{definition}
\label{defi:wordnorm}
Let $G$ be a group and let $S \subset G$ (possibly infinite, possibly non-generating). We define
$$|g|_S := \min \{ k \mid g = s_1 \cdots s_k, s_i \in S \} \in \mathbb{N} \cup \{ \infty \}.$$
\end{definition}

$\Aut$-invariant word norms, that is word norms on $\Aut$-invariant generating sets, have been studied most notably on free and surface groups \cite{norm:free, rank2}, and on graph products of abelian groups \cite{marcinkowski:graph}. The $\Aut$-invariant generating sets are usually obtained by starting with a finite generating set and taking the closure under the action of $\Aut(G)$. In the case of free groups an example of an $\Aut$-invariant generating set is the set of primitive words, and in the case of surface groups one can take the set of simple loops \cite{rank2}. \\

The connection to $\Aut$-invariant homogeneous quasimorphisms is contained in the following proposition:

\begin{proposition}
\label{prop:wordnorm}

Let $S \subset G$ be a subset, and let $S^{\Aut}$ be the closure of $S$ under the action of $\Aut(G)$. Let $f \in Q_h(G)$ be an $\Aut$-invariant homogeneous quasimorphism. Then for all $g \in G$ we have
$$|f(g)| \leq \left( \sup\limits_{s \in S} |f(s)| + D(f) \right) \cdot |g|_{S^{Aut}}.$$
In particular, if $S$ is finite, and $G$ admits an $\Aut$-invariant homogeneous quasimorphism that is nontrivial on $\langle S^{Aut} \rangle$, then $|\cdot|_{S^{\Aut}}$ is unbounded on $\langle S^{Aut} \rangle$.
\end{proposition}

\begin{proof}
The conclusion is obvious if $|g|_{S^{Aut}} = \infty$. Otherwise, let $|g|_{S^{Aut}} = k$, and write $g = \phi_1(s_1) \cdots \phi_k(s_k)$, where $\phi_i \in \Aut(G)$ and $s_i \in S$. Then
\begin{align*}
    |f(g)| &= \left| f\left( \prod\limits_{i = 1}^k \phi_i(s_i) \right) \right| \leq \left| \sum\limits_{i = 1}^k f(\phi_i(s_i)) \right| + (k-1)D(f) \\
    &\leq k \sup\limits_{s \in S} |f(s)| + (k-1) D(f) \leq \left( \sup\limits_{s \in S} |f(s)| + D(f) \right) \cdot |g|_{S^{Aut}}.
\end{align*}
Any homogeneous quasimorphism on $G$ that is nontrivial on $\langle S^{Aut} \rangle$ is unbounded on $\langle S^{Aut} \rangle$, so the result follows.
\end{proof}

Another $\Aut$-invariant word norm uses the following set of elements:

\begin{definition}
\label{deif:autocommutator}

Let $G$ be a group. An \emph{autocommutator} is an expression of the form $[\phi, g] := \phi(g) g^{-1}$, where $g \in G, \phi \in \Aut(G)$. The \emph{autocommutator subgroup} is the group generated by all autocommutators, and is denoted by $[\Aut(G), G]$.
\end{definition}

This notion is commonly attributed to Hegarty after \cite{autocommutators2, autocommutators3}, although it was already considered by Miller with the name \emph{automorphism commutator} \cite{autocommutators1}. Note that $[G, G] \leq [\Aut(G), G]$, but in many cases $[\Aut(G), G]$ is much larger. For instance, if $G = F_n$ is a non-abelian free group with basis $\langle a_1, \ldots, a_n \rangle$, then the autocommutator subgroup is equal to the whole group. To see this, it suffices to notice that $a_1 = [\phi, a_2]$, where $\phi \in \Aut(F_n)$ is any automorphism that sends $a_2$ to $a_1 a_2$.

\begin{definition}
\label{defi:aut:scl}
The \emph{autocommutator length}, denoted $\acl$ (or $\acl_G$) is the word norm on autocommutators, in the sense of Definition \ref{defi:wordnorm}. We define $\acl(g) = \infty$ if $g \notin [\Aut(G), G]$.

The \emph{stable autocommutator length}, denoted $\sacl$ (or $\sacl_G$), is the stabilization of autocommutator length. More precisely, we define $\sacl(g) := \lim\limits_{n \to \infty} \frac{\acl(g^n)}{n} \in \mathbb{R}_{\geq 0}$ in case $g \in [\Aut(G), G]$; $\sacl(g) = \frac{1}{n} \sacl(g^n)$ in case $g^n \in [\Aut(G), G]$ for some $n \geq 1$; and $\sacl(g) = \infty$ in case $g^n \notin [\Aut(G), G]$ for all $n \geq 1$.
\end{definition}

A related notion is that of \emph{mixed stable commutator length}, introduced in \cite{mixed:scl1}. This is defined for a group $G$ and a normal subgroup $N$, as the stable word length in terms of commutators of the form $[g, n], g \in G, n \in N$, and is denoted $\scl_{G, N}$ (this is extended to a, possibly infinite-valued, norm on all of $G$ using the same conventions as in Definition \ref{defi:aut:scl}). This quantity was further studied in \cite{mixed:scl2, mixed:scl3, mixed:scl4}; see \cite{survey} for a survey.

In \cite{karlhofer:free, karlhofer:graph}, the author uses mixed stable commutator length to define $\Aut$-invariant stable commutator length. This is simply the mixed stable commutator length $\scl_{\Aut(G), \Inn(G)}$, under the assumption that $G = \Inn(G)$. It is easy to see that in this case, a commutator $[\Aut(G), \Inn(G)]$ coincides with an autocommutator under the identification of $G$ with $\Inn(G)$ via $\ad$. Therefore stable autocommutator length gives an alternative definition of $\Aut$-invariant stable commutator length, that does not require any assumptions on the center.

There is however a way to describe stable autocommutator length as a mixed stable commutator length. Namely, let $G$ be a group, and consider the semidirect product $G \rtimes \Aut(G)$. Then $\sacl$ on $G$ coincides with the restriction of $\scl_{G \rtimes \Aut(G), G}$ to $G$. Definition \ref{defi:aut:scl} has the advantage of being more intrinsic. \\

A further use of stable autocommutator length is that it gives a lower bound to all mixed stable commutator lengths:

\begin{lemma}
\label{lem:mixed}

Let $G$ be a group and $N$ a normal subgroup. Then $\sacl_N(n) \leq \scl_{G, N}(n)$ for all $n \in [G, N]$.
\end{lemma}

\begin{proof}
It suffices to notice that $[g, m] = [\phi, m]$, where $\phi \in \Aut(N)$ is the automorphism defined by conjugation by $g$.
\end{proof}

In particular, whenever stable autocommutator length on $N$ is unbounded, so is the mixed stable commutator length $\scl_{G, N}$, for any group $G$ that contains $N$ as a normal subgroup.

Finally, we show the analogue of Proposition \ref{prop:wordnorm} in this case:

\begin{proposition}
\label{prop:aut:scl}

Let $f$ be an $\Aut$-invariant homogeneous quasimorphism. Then for all $g \in G$ we have:
$$|f(g)| \leq 2 D(f) \cdot \sacl(g).$$
\end{proposition}

\begin{proof}
We start by noticing that
$$|f([\phi, g])| = |f(\phi(g) \cdot g^{-1})| \leq |f(\phi(g)) + f(g^{-1})| + D(f) = D(f).$$
Now if $S$ is the set of autocommutators in $G$, then $S = S^{\Aut}$, by the following formula:
$$\psi([\phi, g]) = \psi \phi(g) \cdot \psi(g)^{-1} = [\psi \phi \psi^{-1}, \psi(g)].$$
Thus, Proposition \ref{prop:wordnorm} yields $|f(g)| \leq 2 D(f) \cdot \acl(g)$. Using homogeneity of $f$ and stabilizing $\acl(g)$, we conclude.
\end{proof}

\begin{corollary}
\label{cor:aut:scl}
    Let $G$ be a finitely generated group, and suppose that there exists an infinite-dimensional space of $\Aut$-invariant homogeneous quasimorphisms on $G$. Then $\sacl$ is finite and unbounded on $[G, G]$.
\end{corollary}

\begin{proof}
    The fact that $\sacl$ takes only finite values on $[G, G]$ is immediate from the inclusion $[G, G] \leq [\Aut(G), G]$. Since $G$ is finitely generated, it has a finite-dimensional space of homomorphisms. Therefore there exists an $\Aut$-invariant homogeneous quasimorphism $f$ that does not vanish identically on $[G, G]$ (see \cite[Lemma 2.24]{calegari}). Let $g \in [G, G]$ be such that $f(g) \neq 0$. Then Proposition \ref{prop:aut:scl} implies that $\sacl(g^n) \to \infty$ as $n \to \infty$.
\end{proof}

In fact, using the interpretation of stable autocommutator length in terms of mixed stable commutator length on $G \rtimes \Aut(G)$, something much stronger can be shown:

\begin{theorem}[Bavard duality for $\sacl$ \cite{mixed:scl2}]
\label{thm:bavard}

Let $G$ be a group, and let $g \in [\Aut(G), G]$. Then
$$\sacl(g) = \sup \frac{|f(g)|}{2D(f)},$$
where the supremum runs over all $\Aut$-invariant homogeneous quasimorphisms $f$ of positive defect. \qed
\end{theorem}

Corollary \ref{cor:aut:scl} shows how Corollary \ref{intro:cor:wordnorm} follows from Theorem \ref{intro:thm:main}. For groups as in Theorem \ref{intro:thm:main}, it would be interesting to gain more precise information on $\sacl$ than just unboundedness. For instance, for which elements does $\sacl$ vanish (this is discussed more in Section \ref{s:dist})? Is there a spectral gap in $\sacl$? Is $\sacl$ rational? What is the spectrum of $\sacl$ over finitely presented groups? See \cite{heuer, whatis} for a survey of these questions for $\scl$.

\section{The Bestvina--Fujiwara machinery}
\label{s:ggt}

\begin{remark}
    The arguments in this section rely on \cite{BF}. Therefore we will follow the notation therein, and use $h$ to denote quasimorphisms, and $f$ to denote a WPD element. We will get back to our standard notation, using $f$ for quasimorphisms, starting from the next section.
\end{remark}

A group action on a space $X$ is \emph{acylindrical} if for all $R \geq 0$ there exist $N >0, L >0$ such that if two points $x,y \in X$ are at distance at least $L$ then the set \[ \{g \in G : d(x,gx), d(y,gy) \leq R \} \] has size at most $N$. Roughly speaking, coarse stabilizers of points sufficiently far apart have at most $N$ elements in common. A group $G$ is \emph{acylindrically hyperbolic} it admits a non-elementary acylindrical action on some geodesic (Gromov-)hyperbolic metric space $X$ \cite{osin}. In this case, one can also take $X$ to be a graph (more precisely, a Cayley graph of $G$ with respect to an infinite generating set \cite[Theorem~1.2]{osin}).  

We say that an element $f \in G$ is \emph{WPD} (weakly properly discontinuous) for the action of $G$ on $X$ if for every $R \geq 0$ and every $x \in X$ there exists $n$ such that the set \[  \{g\in G : d(x,gx), d(f^nx,gf^nx) \leq R \} \] is finite. We say an action of $G$ on a hyperbolic space $X$ is WPD if every loxodromic element is WPD. An element $f \in G$ is \emph{achiral} if $f^k$ is conjugate to $f^{-k}$ for some $k$, and \emph{chiral} otherwise. If $f$ is loxodromic for a WPD action, chirality is a condition that ensures conjugation cannot flip the quasi-axis of $f$.

When $f$ is acting as a loxodromic element on a hyperbolic metric space, the WPD condition says that the action looks acylindrical `along the axis of $f$.' In particular, acylindrical actions are also WPD. Furthermore, if the action of $G$ on $X$ is non-elementary and contains a WPD element then $G$ is acylindrically hyperbolic (this is far from trivial and was first proved in \cite{DGO}, see \cite{spinning} for a more recent approach).

Bestvina and Fujiwara showed that if a group $G$ admits a non-elementary action on a hyperbolic graph with a WPD element, then $G$ has many homogeneous quasimorphisms \cite{BF}. They do this by showing that if $f$ is a chiral loxodromic WPD element, Brooks' construction can be `quasified' to produce a quasimorphism that is unbounded on $\langle f \rangle$ (note that if $f$ is achiral, then every homogeneous quasimorphism vanishes on $f$).
To find spaces of quasimorphisms, they produce many chiral loxodoromic elements whose associated quasimorphisms are suitably independent. We observe below that their constuction affords a great deal of freedom when it comes to controlling the supports of these quasimorphisms. Due to reliance on the actual proofs in \cite{BF}, our proof is far from self contained and we recommend the reader has \cite{BF} at hand while reading it.


\begin{theorem}
\label{thm:BF}
Suppose that the action of a group $G$ on a hyperbolic graph $X$ is non-elementary and WPD. \begin{enumerate} \item If $f \in G$ is a chrial loxodromic element for the action on $X$, then there exists a homogeneous quasimorphism $h \colon G \to \mathbb{R}$ such that $h(f) \neq 0$. \item If $H \leq G$ is non-elementary with respect to the action on $X$, then the restriction $Q_h(G) \to Q_h(H)$ has an infinite-dimensional image.\end{enumerate}
\end{theorem}

Roughly speaking, Theorem~\ref{thm:BF} tells us that we can build many homogeneous quasimorphisms on $G$ that do not vanish on $H$.

\begin{proof}
  Following \cite{BF}, given loxodromic elements $g_1, g_2$ we say that $g_1 \sim g_2$ if some positive powers of $g_1$ and $g_2$ are conjugate in $G$ (this is equivalent to the initial definition of $\sim$ in \cite{BF} by Part (4) of Proposition 6 in the same paper). An element $f\in G$ is chiral if and only if $f \not \sim f^{-1}$. In Proposition 5 of Section 2 of \cite{BF}, it is shown that for any chiral loxodromic element $f$, there exists a quasimorphism $h_{f^a}:G\to \mathbb{R}$ such that $h_{f^a}(f^n)$ grows linearly with $n$.
  Its homogenisation is then nonzero on $\langle f\rangle $ (to be slightly more precise in our application of \cite[Proposition~5]{BF}, after replacing $f$ with a power, we can assume $f$ is a basis element of a free Schottky subgroup $F<G$, so will be cyclically reduced in $F$).

  For the second assertion, Part (5) of \cite[Proposition~6]{BF} explains how any free subgroup $F \leq G$ consisting entirely of loxodromic elements has $g_1, g_2 \in F$ such that $g_1 \not\sim g_2$ and $g_1$, $g_2$ are independent (roughly speaking, an infinite sequence of loxodromic elements whose axes all have a sufficiently long overlap must contain infinitely many conjugacy classes in $G$ because of weak proper discontinuity).  As the action of $H$ on $X$ is non-elementary, we may take $F$ to be a subgroup of $H$ and therefore obtain loxodromic elements $g_1, g_2 \in H$ such that $g_1 \not\sim g_2$.

Given that $g_1, g_2$ are independent loxodromics with $g_1 \not\sim g_2$, in Section 2 of \cite{BF}, Bestvina and Fujiwara use a variation on Brooks' construction to find a sequence $f_1, f_2, \ldots$ of chiral elements in $\langle g_1, g_2 \rangle$ and quasimorphisms $h_1, h_2, \ldots \in Q(G)$ such that  each $h_i$ is unbounded on $\langle f_i \rangle$ and vanishes on $\langle f_1, \ldots, f_{i-1} \rangle$ \cite[Proof of Theorem~1]{BF}. This implies that any nontrivial linear combination $\sum a_ih_i$ is unbounded, and therefore the homogenisations of the $h_i$ (Proposition \ref{prop:homogenisation}) are linearly independent in $Q_h(G)$. Moreover, the linear independence of the $h_i$ is witnessed already by restricting to the group generated by the $f_i$, so that all of these quasimorphisms remain nontrivial under the restriction map 
 $Q_h(G) \to Q_h(H)$.
\end{proof}

A subgroup $N$ of a group $G$ is called \emph{$s$-normal} if every intersection $N \cap gNg^{-1}$ of $N$ with a conjugate of itself is infinite. The simplest examples of $s$-normal subgroups are infinite normal subgroups.

\begin{theorem}[\cite{osin}, Lemma~7.1] \label{thm:osin}
Let $G$ be a group acting acylindrically and non-elementarily on a hyperbolic space
$X$. Then every $s$-normal subgroup of G acts non-elementarily.
\end{theorem}

\begin{corollary}\label{cor:BF}
If $H \leq G$ is an $s$-normal subgroup of an acylindrically hyperbolic group $G$ then the restriction $Q_h(G) \to Q_h(H)$ has an infinite-dimensional image.
\end{corollary}

\begin{proof}
By Theorem~\ref{thm:osin}, we can apply Theorem~\ref{thm:BF} to any acylindrical action of $G$.
\end{proof}

\section{Proofs of the main results}
\label{s:proof}

\subsection{Acylindrically hyperbolic automorphism groups}

We are finally ready to prove Theorem \ref{intro:thm:ah}, which we recall for the reader's convenience:

\begin{theorem}
\label{thm:ah}

Let $G$ be a group, and suppose that $G$ is not virtually central and $\Aut(G)$ is acylindrically hyperbolic. Then there exists an infinite-dimensional space of $\Aut$-invariant homogeneous quasimorphisms on $G$.
\end{theorem}

\begin{proof}
If $G$ is not virtually $Z(G)$, then $G/Z(G) \cong \Inn(G)$ is an infinite normal subgroup of $\Aut(G)$. Therefore, as $\Aut(G)$ is acylindrically hyperbolic, Corollary~\ref{cor:BF} tells us that the restriction $Q_h(\Aut(G)) \to Q_h(\Inn(G))$ has infinite-dimensional image. Lemma \ref{lem:main} tells us that the image of the composition 
\[ Q_h(\Aut(G)) \to Q_h(\Inn(G)) \to Q_h(G) \]
consists of $\Aut(G)$-invariant homogeneous quasimorphisms, and as $G \to G/Z(G)\cong \Inn(G)$ is surjective, the second arrow above is injective.
\end{proof}

\begin{remark}
Note that the assumption that $G$ is not virtually central does not follow from the acylindrical hyperbolicity of $\Aut(G)$. For instance, the automorphism group of $\mathbb{Z}^2$ is isomorphic to $\mathrm{GL}_2(\mathbb{Z})$, which is virtually free non-abelian, and thus acylindrically hyperbolic.
\end{remark}

\begin{remark}\label{rem:subgroups}
If $S$ is a finite set and $S^{\Aut}$ does not generate $G$, then $H=\langle S^{\Aut} \rangle$ is a characteristic subgroup of $G$. If $H$ is not virtually central in $G$ and $\Aut(G)$ is acylindrically hyperbolic, then the same proof as in Theorem~\ref{thm:ah} implies that there are nontrivial homogeneous quasimorphisms on $H$. Proposition~\ref{prop:wordnorm} also implies that the word norm $|\cdot|_{S^{\Aut}}$ is unbounded (but finite) on $H$.
\end{remark}

We immediately obtain Corollary \ref{intro:cor:bin}:

\begin{corollary}
\label{cor:bin}

Let $Q$ be a countable group. Then there exists a finitely generated acylindrically hyperbolic group $G$ with an infinite-dimensional space of $\Aut$-invariant quasimorphisms such that $\Out(G) \cong Q$.
\end{corollary}

\begin{proof}
    By \cite[Theorem 4.5]{bin}, there exists an acylindrically hyperbolic group $G$ such that $\Out(G) \cong Q$ and $\Aut(G)$ is acylindrically hyperbolic. We conclude by Theorem \ref{thm:ah}.
\end{proof}

\subsection{Automorphisms of graph products}

Suppose that $\G$ is a finite graph and we have a set of groups $\mathcal{G}=\{ G_v\}_{v \in V(\G)}$ indexed by the vertices of $\G$. The \emph{graph product} $\G\mathcal{G}$ is defined to be the quotient of the free product of the elements of $\mathcal{G}$ by the normal subgroup generated by the set of subgroups $\{ [G_v,G_w] : \{v,w\} \in E\G\}$. In other words, $\gp$ is the `most free' group containing the $G_v$ such that the groups at adjacent vertices in $\G$ commute. Examples of graph products include right-angled Artin groups (where we take each $G_v=\mathbb{Z}$) and right-angled Coxeter groups (where we take each $G_v = \mathbb{Z}/2\mathbb{Z}$). 

\begin{theorem}\label{thm:graph_product}
Suppose that $\gp$ is a graph product of nontrivial finitely generated groups $\mathcal{G}=\{G_v\}_v$ on a finite graph $\G$ such that: 
\begin{itemize}
    \item No vertex group decomposes nontrivially as a graph product;
    \item $\G$ is not complete;
    \item $\G$ is not the iterated join of a complete graph with pairs of isolated vertices, such that all of the attached isolated vertices are labelled by $\mathbb{Z}/2\mathbb{Z}$.
\end{itemize}
Then the space of $\Aut$-invariant homogeneous quasimorphisms on $\gp$ is infinite-dimensional.
\end{theorem}
The last condition appears as two isolated vertices labelled by $\mathbb{Z}/2\mathbb{Z}$ gives the graph product isomorphic to the infinite dihedral group $D_\infty=\mathbb{Z}/2\mathbb{Z} \ast \mathbb{Z}/2\mathbb{Z}$. As these dihedral groups are virtually $\mathbb{Z}$, the group $\gp$ is virtually isomorphic to a group $G \oplus \mathbb{Z}^n$ in the third bullet point (where $G$ is a graph product on a complete graph).

\begin{proof}
Any finite graph $\G$ decomposes as a join $\G=\G_0 \ast \G_1 \ast \cdots \ast \G_n$ such that\begin{itemize}
    \item $\G_0$ is complete
    \item Each $\G_i$ contains at least two vertices and does not decompose as a nontrivial join. 
\end{itemize}
This decomposition is unique up to permuting the set $\{\G_i\}_{i\geq 1}$. Let $H_i$ be the subgroup of $\gp$ generated by the vertex groups of $\G_i$. The main theorem of \cite{auto:graph}, proves that for all $i \geq 1$, either $\Aut(H_i)$ is acylindrically hyperbolic or $H_i \cong D_\infty$. By our hypothesis, after possibly relabeling we may assume that $H_1 \neq D_\infty$, so that $\Aut(H_1)$ is acylindrically hyperbolic. Our goal is to extend $\Aut(H_1)$-invariant homogeneous quasimorphisms on $H_1$ to $\Aut(\gp)$-invariant homogeneous quasimorphisms of $\gp$. First, let \[ H=H_1 \oplus H_2 \oplus \cdots \oplus H_n, \] and without loss of generality suppose that $H_1, \ldots, H_k$ are the factors of $H$ that are isomorphic to $H_1$. Let $J \leq S_n$ be the subgroup of the symmetric group consisting of permutations $\sigma$ such that $H_{\sigma(i)}\cong H_i$ for all $i=1, \ldots, n$. Such a permutation restricts to a permutation of $\{1,\ldots, k\}$. If we fix compatible isomorphisms between isomorphic factors, these induce an embedding $\sigma \mapsto \phi_\sigma$ from $J$ to $\Aut(H)$ such that \[ \phi_\sigma(h_1,\ldots, h_n)= (h_{\sigma^{-1}(1)}, \ldots, h_{\sigma^{-1}(n)}). \]
To keep our notation consistent with \cite{auto:graph}, we label the image of $J$ in $\Aut(H)$ by $S$. We want to promote quasimorphisms $f$ on $H_1$ to quasimorphisms on $H$ that are invariant under $S$, so we need to `average out' over the factors isomorphic to $H_1$. To this end, for $h=(h_1, \ldots, h_n) \in H$ and a map $f \colon H_1 \to \mathbb{R}$ define \[ \int f := \sum_{i=1}^k f (h_i). \]

\begin{claim}
For $H<\gp$ as above,  $f \mapsto \int f$ sends $\Aut(H_1)$-invariant maps $f\colon H_1 \to \mathbb{R}$ to $\Aut(H)$-invariant maps $\int f \colon H \to \mathbb{R}$. Furthermore, the map $\int:\Map(H_1,\mathbb{R}) \to \Map(H,\mathbb{R})$ is injective and restricts to an injective map $Q_h(H_1) \to Q_h(H)$.
\end{claim}

\begin{proof}[Proof of Claim]
Proposition~3.11 of \cite{auto:graph} implies that $\Aut(H)$ is generated by the automorphisms of the individual factors $H_1, \ldots, H_n$ and the group $S$. Hence to prove invariance, we need only to look at elements of $S$ and automorphisms of $H$ that preserve each factor. If $\phi_\sigma \in S$ and $f\colon H \to \mathbb{R}$ then
\begin{align*} \int f(\phi_\sigma(h))&=\sum_{i=1}^kf(h_{\sigma^{-1}(i)}) \\ &=\sum_{i=1}^k f (h_i) = \int f (h)
\end{align*}
as $\{h_1,\ldots,h_k\}=\{h_{\sigma^{-1}(1)},\ldots, h_{\sigma^{-1}(k)}\}$. Furthermore, if $f: H_1 \to \mathbb{R}$ is $\Aut(H_1)$-invariant, and $\phi$ is an automorphism of $H$ that restricts to an automorphism $\phi_i$ of each factor, then  \begin{align*} \int f(\phi(h))&=\sum_{i=1}^kf(\phi_i(h_i)) \\ &=\sum_{i=1}^k f (h_i) = \int f (h).
\end{align*}
Hence if $f$ is $\Aut(H_1)$-invariant, the map $\int f$ is $\Aut(H)$-invariant. The averaging map $\int$ is injective as $\int f |_{H_1}=f$, and if $f$ is a quasimorphism with defect $D$ then the defect of $\int f$ is at most $kD$ (in fact, using \cite[Lemma 2.24]{calegari}, the defect of $\int f$ is exactly $kD$).
One can also check that if $f$ is homogeneous then so is $\int f$. Hence we obtain an injective map \[ \int \colon Q_h(H_1)^{\Aut(H_1)} \to Q_h(H)^{\Aut(H)} \]
\end{proof}

The group $H_1$ is infinite and its center $Z(H_1)$ is trivial (as $\G_1$ does not decompose as a nontrivial join). Therefore, as $\Aut(H_1)$ is acylindrically hyperbolic \cite[Theorem~7.1]{auto:graph}, Theorem \ref{thm:ah} implies that $Q_h(H_1)$ contains an infinite-dimensional space of $\Aut(H_1)$-invariant homogeneous quasimorphisms. The above claim implies that these give us an infinite-dimensional subspace of $\Aut(H)$-invariant homogeneous quasimorphisms on $H$. The group $H_0 \leq \gp$ is characteristic in $\gp$ (again this is contained in Proposition~3.11 of \cite{auto:graph}).
Proposition \ref{prop:char} then implies that the map $Q_h(H) \to Q_h(\gp)$ induced by the quotient \[\gp \to \gp / H_0 \cong H\] gives us an infinite-dimensional space of $\Aut(\gp)$-invariant homogeneous quasimorphisms of $\gp$ coming from the $\Aut(H)$-invariant homogeneous quasimorphisms of $H$. 
 \end{proof}

\begin{corollary}
\label{cor:gp:fga}

If $\gp$ is a graph product of nontrivial finitely generated abelian groups on a finite graph $\Gamma$, and $\gp$ is not virtually abelian, then the space of $\Aut$-invariant homogeneous quasimorphisms on $\gp$ is infinite-dimensional.
\end{corollary}

\begin{proof}
Suppose the space of $\Aut$-invariant homogeneous quasimorphisms on $\gp$ is finite dimensional. As all vertex groups are finitely generated and abelian we may refine $\G$ so that every vertex group is cyclic and does not decompose as a graph product (i.e. for every vertex $v$ either $G_v=\mathbb{Z}/p^k\mathbb{Z}$ for a prime $p$ and $k \in \mathbb{N}$ or $G_v = \mathbb{Z}$). By Theorem~\ref{thm:graph_product}, $\G$ is either complete or the iterated join of a complete graph, with pairs of isolated vertices, such that the pairs of isolated vertices are labelled by $\mathbb{Z}/2\mathbb{Z}$. Then $\gp \cong A \oplus D_\infty^n$ for an abelian group $A$ and some $n$, and is virtually isomorphic to $A \oplus \mathbb{Z}^n$.
\end{proof}

\begin{remark}
In the context of Corollary \ref{cor:gp:fga}, Marcinkowski had already proved that $\gp$ is virtually abelian if and only if the $\Aut$-invariant word norm is bounded \cite{marcinkowski:graph}. In virtue of Proposition \ref{prop:wordnorm}, Corollary \ref{cor:gp:fga} implies that this in turn is equivalent to the non-existence of unbounded $\Aut$-invariant homogeneous quasimorphisms.
\end{remark}

\begin{corollary}
If $G$ is a finitely generated right-angled Artin or Coxeter group that is not virtually abelian, then the space of $\Aut$-invariant homogeneous quasimorphisms on $G$ is infinite-dimensional. \qed
\end{corollary}

\subsection{The proof of Theorem~\ref{intro:thm:main}}

Theorem~\ref{intro:thm:main} from the introduction follows quickly by combining the above work with results of Genevois and Horbez \cite{auto:infend}. We repeat the statement and give a brief proof below.

\begin{theorem}
\label{thm:main}

Let $G$ be a finitely generated group, and suppose that one of the following holds:
\begin{enumerate}
    \item $G$ is a non-elementary Gromov-hyperbolic group;
    \item $G$ is non-elementary hyperbolic relative to a collection of finitely generated subgroups $\mathcal{P}$, and no  group in $\mathcal{P}$ is relatively hyperbolic;
    \item $G$ has infinitely many ends;
    \item $G$ is a graph product of finitely generated abelian groups on a finite graph, and $G$ is not virtually abelian.
\end{enumerate}

Then there exists an infinite-dimensional space of $\Aut$-invariant homogeneous quasimorphisms on $G$.
\end{theorem}

\begin{proof}
In Items $1, 2$ and $3$, this is an application of Theorem \ref{thm:ah} (Theorem \ref{intro:thm:ah}) and the fact that $\Aut(G)$ is acylindrically hyperbolic \cite[Theorem 1.1, Theorem 1.3]{auto:infend} (see also \cite[Theorem 1.2]{auto:oneend}). Item $4$ is Corollary \ref{cor:gp:fga}.
\end{proof}

\subsection{Free products}

Theorem \ref{thm:main} relies on \cite{auto:infend}, which is why we require finite generation. However it is known that a free product of finitely many (possibly infinitely generated) freely indecomposable groups admits an infinite-dimensional space of $\Aut$-invariant homogeneous quasimorphisms, provided at most two factors are infinite cyclic \cite[Theorem 1.1]{karlhofer:graph}. This restriction is due to the fact that free groups must be dealt with separately, and when \cite{karlhofer:graph} was posted, only the rank $2$ case was solved \cite{rank2}. Using Theorem \ref{thm:main} for free groups, we are able to complete the picture:

\begin{corollary}
\label{cor:freeprod}

For $k \geq 2$, let $G = G_1 * \cdots * G_k$ be a free product of non-trivial freely indecomposable groups, and suppose that $G$ is not isomorphic to the infinite dihedral group. Then $G$ admits an infinite-dimensional space of $\Aut$-invariant homogeneous quasimorphisms, each of which is trivial when restricted to each $G_i$.
\end{corollary}

For the proof, we follow \cite{karlhofer:graph} and denote by $\Aut^0(G)$ the subgroup of $\Aut(G)$ generated by factor automorphisms, transvections and partial conjugations (see e.g. \cite[Section 4.3]{karlhofer:graph} for definitions, which will not be needed for the proof). By \cite{auto:free1, auto:free2}, $\Aut(G)$ is generated by $\Aut^0(G)$ together with \emph{swap automorphisms}, i.e. automorphisms given by permuting two isomorphic factors.

\begin{proof}
    If at most one factor is isomorphic to $\mathbb{Z}$, this is covered by \cite[Theorem 1.1]{karlhofer:graph}. Otherwise, up to reordering we can write $G = G_1 * \cdots * G_m * F_n$, where $n \geq 2, m \geq 0$ and none of the $G_i$ is isomorphic to $\mathbb{Z}$.
    
    Let $p : G \to F_n$ be the projection. By \cite[Lemma 5.2]{karlhofer:graph}, $\ker(p)$ is $\Aut^0$-invariant (although the statement is more restrictive, the exact same proof applies to our case). Moreover, any two isomorphic factors are either among the $G_1, \ldots, G_m$ or among the free factors of $F_n$, so swap automorphisms also leave $\ker(p)$ invariant. It follows that $\ker(p)$ is characteristic in $G$, and thus by Proposition \ref{prop:char} the pullback of an unbounded $\Aut$-invariant homogeneous quasimorphism of $F_n$ (provided by Theorem \ref{thm:main}) is an unbounded $\Aut$-invariant homogeneous quasimorphism of $G$. Such a quasimorphism obviously vanishes on the factors $G_i$, and moreover it vanishes on the basis elements of $F_n$, by Lemma \ref{lem:camille}.
\end{proof}

\begin{remark}
Note that our proof is straightforward only because we can assume that at least two factors are isomorphic to $\mathbb{Z}$. For instance, in the case of the free product of two infinitely generated groups, the construction of \cite{karlhofer:free} is necessary in order to produce $\Aut$-invariant quasimorphisms.
\end{remark}

In light of Theorem \ref{thm:graph_product} and Corollary \ref{cor:freeprod} the following question is natural:

\begin{question}
\label{q:graph_product}

Let $\Gamma \mathcal{G}$ be a graph product of nontrivial (possibly infinitely generated) groups on a finite graph satisfying the hypotheses of Theorem \ref{thm:graph_product}. Is the space of $\Aut$-invariant homogeneous quasimorphisms on $\Gamma \mathcal{G}$ infinite-dimensional?
\end{question}

\section{Outlook}
\label{s:outlook}

In this last section, we discuss two directions for future research.

\subsection{Acylindrically hyperbolic groups}

A positive answer to the following question would give a natural extension of some of our main results:

\begin{question}
\label{q2}

Let $G$ be a finitely generated acylindrically hyperbolic group. Does $G$ admit an infinite-dimensional space of $\Aut$-invariant homogeneous quasimorphisms?
\end{question}

In \cite[Question 1.1]{auto:oneend}, the author asks whether $\Aut(G)$ is acylindrically hyperbolic, for every finitely generated acylindrically hyperbolic group $G$. In light of Theorem \ref{intro:thm:ah}, an affirmative answer to Genevois's question implies an affirmative answer to Question \ref{q2}. Nevertheless, Question \ref{q2} is strictly weaker: we have seen in Theorem \ref{thm:graph_product} that it is possible for a group $G$ to have an infinite-dimensional space of $\Aut$-invariant homogeneous quasimorphisms, even when $\Aut(G)$ is not acylindrically hyperbolic \cite[Theorem 1.5]{auto:graph}. \\

It is a well-known open problem in geometric group theory to determine whether acylindrical hyperbolicity is a commensurability invariant \cite[Question 2]{erratum}.
A milder question could be to ask whether groups that are commensurable to acylindrically hyperbolic groups have an infinite-dimensional space of quasimorphisms. This is related to $\Aut$-invariant homogeneous quasimorphisms in the following way:

\begin{lemma}
\label{lem:commensurable}

Suppose that Question \ref{q2} has an affirmative answer. Then every group that is commensurable to a finitely generated acylindrically hyperbolic group has an infinite-dimensional space of homogeneous quasimorphisms.
\end{lemma}

The proof will make use of basic notions from bounded cohomology. Since these concepts are used exclusively in this proof, we prefer not to introduce them fully, and rather refer the reader to the literature \cite{frigerio, monod} for more information.

\begin{proof}
Let $G_1$ be a finitely generated acylindrically hyperbolic group. Let $G_2$ be a group commensurable to $G_1$, thus there exist isomorphic finite-index subgroups $K_i < G_i$. Since finite-index subgroups of acylindrically hyperbolic groups are acylindrically hyperbolic, $K_1$ is acylindrically hyperbolic, and thus so is $K_2$. Up to passing to a finite-index normal subgroup, we may assume that $K_2$ is normal in $G_2$.

Now the Lyndon--Hochschild--Serre exact sequence in bounded cohomology, together with amenability of the quotient $G_2/K_2$, implies that the restriction $\HH^2_b(G_2) \to \HH^2_b(K_2)$ is an isomorphism onto $\HH^2_b(K_2)^{G_2}$ \cite[Example 12.4.3]{monod}. Moreover, since $K_2$ has finite index in $G_2$, the analogous statement holds at the level of cohomology: the restriction $\HH^2(G_2) \to \HH^2(K_2)$ is an isomorphism onto $\HH^2(K_2)^{G_2}$ \cite[Proposition III.10.4]{brown}. It follows that $\EH^2_b(G_2) \cong \EH^2_b(K_2)^{G_2}$, where we recall that $\EH^2_b$ denotes the exact part of bounded cohomology, i.e. the kernel of the comparison map $\HH^2_b \to \HH^2$.

The space of $\Aut$-invariant homogeneous quasimorphisms on $K_2$, modulo homomorphisms, is a subspace of $\EH^2_b(K_2)^{G_2}$. Since $K_2$ is finitely generated it can only have a finite-dimensional space of homomorphisms, therefore the hypothesis implies that $\EH^2_b(K_2)^{G_2}$ is infinite-dimensional. It follows that $\EH^2_b(G_2)$ is also infinite-dimensional, that is, $G_2$ has an infinite-dimensional space of homogeneous quasimorphisms.
\end{proof}

\begin{remark}
Note that the proof of Lemma \ref{lem:commensurable} does not need the full strength of Question \ref{q2}, but rather it only requires that for every group $A \leq \Aut(G)$ whose image in $\Out(G)$ is finite, there exists an infinite-dimensional space of $A$-invariant homogeneous quasimorphisms on $G$. Combining the arguments of Lemmas \ref{lem:main} and \ref{lem:commensurable}, one can show that this holds whenever $A$ itself has an infinite-dimensional space of homogeneous quasimorphisms.
\end{remark}

\subsection{Classifying elements with positive stable autocommutator length}
\label{s:dist}

Given a norm $| \cdot |$ on a group $G$, we say that the cyclic subgroup $\langle g \rangle$ is \emph{undistorted} if $|g^n|$ is finite and grows linearly with $n$. An element of the autocommutator subgroup has positive sacl if and only if $\langle g \rangle$ is undistorted in the norm given by autocommutator length. By Theorem \ref{thm:bavard}, this is equivalent to the existence of an $\Aut$-invariant homogeneous quasimorphism $f$ such that $f(g) \neq 0$. In this case, $\langle g \rangle$ is also undistorted for other mixed stable commutator lengths as in Lemma \ref{lem:mixed}, as well as $\Aut$-invariant word norms as in Proposition \ref{prop:wordnorm} (provided these take finite values on $\langle g \rangle$). \\

Recall that an element of a group $G$ is chiral if $g^k$ is not conjugate to $g^{-k}$ for all $k \in \mathbb{N}$, and an element $g$ of an acylindrically hyperbolic group $G$ is a \emph{generalized loxodromic} element if there exists an action of $G$ on a hyperbolic space $X$ such that $g$ is WPD for that action. By \cite[Theorem~1.2]{osin}, such an action may also be assumed to be acylindrical. The first part Theorem~\ref{thm:BF}, and the discussion in Section~\ref{s:norms}, imply the following.

\begin{proposition}
Let $G$ be a finitely generated group such that $\Aut(G)$ is acylindrically hyperbolic, and let $g \in [\Aut(G), G]$. If $\ad_g$ is a chiral generalized loxodromic element of $\Aut(G)$ then: \begin{itemize}
    \item There exists an $\Aut$-invariant homogeneous quasimorphism $f$ such that $f(g) \neq 0$, equivalently $\sacl(g) > 0$;
    \item The subgroup $\langle g \rangle$ is undistorted for any $\Aut$-invariant norm $|\cdot|_{S^{\Aut}}$ generated by a finite subset $S \subset G$.
\end{itemize}
\end{proposition}

In general, it is a hard problem to classify the generalized loxodromic elements of an acylindrically hyperbolic group. For instance, even though it is known that the generalized loxodromic elements of mapping class groups are exactly the pseudo-Anosov elements \cite[Example~6.5]{osin}, classifying the generalized loxodromic elements in $\Out(F_n)$ is an open problem. However, we do know which inner automorphisms are generalized loxodromics in $\Aut(F_n)$:

\begin{theorem}[\cite{auto:infend}]
Let $g \in F_n$. Then $\ad_g$ is a generalized loxodromic element of $\Aut(F_n)$ if and only if $g$ is not elliptic in any $\Zmax$ splitting of $F_n$.
\end{theorem}

Recall that a \emph{splitting} is a nontrivial minimal action of $F_n$ on a simplicial tree. A splitting is $\Zmax$ if all edge stabilizers are maximal cyclic, and is \emph{free} if all edge stabilizers are trivial.

\begin{proof}
The fact that such an element is a generalized loxodromic element is a consequence of the proof of Theorem~5.8 in \cite{auto:infend}. Conversely, if $g$ is elliptic in a $\Zmax$ splitting then the centralizer of $\ad_g$ in $\Aut(F_n)$ contains a copy of $\mathbb{Z}^2$ (the stabilizer of the splitting in $\Out(F_n)$ contains an infinite group of \emph{twists} that fix the conjugacy class of $g$: see e.g. Section~2.6 of \cite{GL}),  so this element is not a generalized loxodromic, as the centralizer of a WPD element must be virtually cyclic.
\end{proof}

This is interesting as there is a gap between generalized loxodromics and elements of $F_n$ that generate undistorted subgroups with respect to the primitive ($\Aut$-invariant) word norm:

\begin{theorem}[\cite{rank2}, Theorem~3.10]
Let $|\cdot|_{{\rm{Prim}}}$ be the ($\Aut$-invariant) word norm on $F_n$ given by the infinite generating set consisting of primitive elements - which coincides with $S^{\Aut}$, where $S$ is a free basis of $F_n$.  For an element $g \in F_n$, its cyclic subgroup $\langle g \rangle$ is undistorted in $|\cdot|_{{\rm{Prim}}}$ if and only if $g$ is not elliptic in any free splitting of $F_n$ (equivalently, $g$ is not contained in a proper free factor of $F_n$).
\end{theorem}

Brandenbursky and Marcinkowski prove more than this: they show that elements that are not contained in a proper free factor are nonzero under some homogeneous quasimorphism that is uniformly bounded on all primitive elements. Note that this is a strictly larger class of elements than those that appear in the results of Genevois and Horbez. The following lemma, suggested to us by Camille Horbez, tells us that this gap between the Genevois--Horbez results and those of Brandenbursky--Marcinkowski is the only thing to worry about in terms of autocommutator length:

\begin{lemma}\label{lem:camille}
If $g \in F_n$ is contained in a proper free factor or $\ad_g$ is achiral in $\Aut(F_n)$ then $\sacl(g)=0$.
\end{lemma}

\begin{proof}
Let $x$ be a basis element in a complementary factor to one containing $g$. Then there exists an automorphism $\phi$ mapping $x \mapsto g^nx$ and fixing all other elements from an appropriately chosen basis, so that $[\phi,x]=\phi(x)x^{-1}=g^n$ and $\sacl(g)=0$. If $g$ is achiral in $\Aut(F_n)$ then there exists $\phi$ such that $\phi(g)=g^{-1}$. Thus $[\phi, g^n]=g^{-2n}$, so $\sacl(g)=0$, also.
\end{proof}

One possible way of accounting for this is to weaken the WPD condition to WWPD. This was used to great success in \cite{BBF} in the study of stable commutator length on mapping class groups, and leads naturally to the following question:

\begin{question}
If $g \in F_n$ is chiral and not contained in a proper free factor, does there exist an action of $\Aut(F_n)$ on a hyperbolic graph such that $\ad_g$ is WWPD?
\end{question}

A positive answer to this question (combined with the methods in \cite{BBF}) would imply that if $g\in F_n$ is not contained in a proper free factor and $g$ is chiral in $\Aut(F_n)$, then $\sacl(g)>0$ (hence providing a converse to Lemma~\ref{lem:camille}).

\begin{remark}
It was suggested to us by the referee that the WWPD condition might also be useful for tackling situations in Theorem~\ref{intro:thm:main} where peripheral subgroups are not finitely generated. Indeed, the example $G=Z \ast Z$ where $Z$ is an infinite direct sum of $\mathbb{Z}$'s is given by Genevois and Horbez as an example of a relatively hyperbolic group such that $\Aut(G)$ is not acylindrically hyperbolic \cite[Remark~5.9]{auto:infend}. However, the action of $G$ on the usual Bass--Serre tree $T$ of the free product $Z \ast Z$ extends to an action of $\Aut(G)$ (this holds for an arbitrary free product of \emph{two} freely indecomposable groups, as the free product decomposition is essentially unique and the associated Outer space is a single point). Although this action is not acylindrical, one can show that loxodromic elements of $\Inn(G)$ on this tree are WWPD with respect to the (extended) $\Aut(G)$ action on $T$. This could be used to  reprove Karlhofer's result about Aut-invariant quasimorphisms on free products \cite{karlhofer:free}, although the quasimorphisms might be less explicit. More generally, one would hope that the constructions in \cite{auto:infend} will create actions where inner automorphisms are WWPD even when the automorphism groups are not acylindrically hyperbolic.
\end{remark}

\subsection{No $\Aut$-invariant quasimorphisms}

On the opposite end of the spectrum, we can ask about which groups admit unbounded quasimorphisms, but no $\Aut$-invariant ones. In particular, we ask the following question:

\begin{question}
\label{q:fg}

Let $G$ be a finitely generated group with an infinite-dimensional space of homogeneous quasimorphisms. Is it possible that $G$ admits no unbounded $\Aut$-invariant quasimorphism?
\end{question}

If we remove the hypothesis of finite generation, then we do have examples. The following was kindly pointed out to us by the referee:

\begin{example}
Let $F$ be a free group of infinite rank. Then $F$ surjects onto a free group of finite rank, therefore it has an infinite-dimensional space of homogeneous quasimorphisms (even modulo homomorphisms). However, it has no unbounded $\Aut$-invariant quasimorphism. Indeed, the argument from Lemma \ref{lem:camille} applies to $F$ as well, and since \emph{every} element in $F$ is contained in a proper free factor, $\sacl$ vanishes identically on $F$. We conclude by Proposition \ref{prop:aut:scl}.
\end{example}

Restricting to finitely generated groups, we can produce examples, but only with a finite-dimensional space of homogeneous quasimorphisms:

\begin{example}
Let $\overline{T}$ be the lift to the real line of Thompson's group $T$. The group $\overline{T}$ is perfect, and $Q_h(\overline{T})$ is generated by the rotation quasimorphism
$$\rot(g) = \lim\limits_{n \to \infty} \frac{g^n(0)}{n}$$
(see e.g. \cite[Section 5.1.3]{calegari}). Therefore $Q_h(\overline{T})$ is one-dimensional, and does not contain unbounded homomorphisms. Moreover, let $\phi \in \Aut(\overline{T})$ be the automorphism given by the change of orientation, namely $\phi(g)(x) = -g(-x)$ for all $x \in \mathbb{R}$. Then $\rot(\phi(g)) = - \rot(g)$, and therefore $\rot$ is not $\Aut$-invariant.

Similarly, $Q_h(\overline{T}^n)$ is $n$-dimensional, and contains no unbounded homomorphism, and no unbounded $\Aut$-invariant homogeneous quasimorphism. This follows from the fact that the functor $Q_h$ commutes with finite direct products, and $\Aut(\overline{T}^n)$ contains $\Aut(\overline{T})^n$.
\end{example}

Note that the argument is made possible by the very explicit description of $Q_h(\overline{T})$. Therefore it is unlikely that a similar argument will work in the case of Question \ref{q:fg}.

\vspace{0.5cm}

\footnotesize

\bibliographystyle{alpha}
\bibliography{ref}

\vspace{0.5cm}

\normalsize

\noindent{\textsc{Department of Mathematics, ETH Z\"urich, Switzerland}}

\noindent{\textit{E-mail address:} \texttt{francesco.fournier@math.ethz.ch}} \\

\noindent{\textsc{Mathematical Institute, Oxford University, U.K.}}

\noindent{\textit{E-mail address:} \texttt{wade@maths.ox.ac.uk}}

\end{document}